\documentclass[12pt]{amsart}
\usepackage{amsfonts, amsthm, amsmath, amssymb, tikz-cd}
\usepackage{hyperref}
\hypersetup{colorlinks=false}

\usepackage{color}

\def \Hilb {\operatorname{Hilb}}

\newtheorem{theorem}{Theorem}[section]

\newtheorem{lemma}[theorem]{Lemma}

\newtheorem{conjecture}[theorem]{Conjecture}
\theoremstyle{remark}
\newtheorem{remark}[theorem]{Remark}

\theoremstyle{definition}
\newtheorem{defi}[theorem]{Definition}

\begin{document}

\title{Freeness alone is insufficient for Manin-Peyre}

\author{Will Sawin}
\address{Columbia University \\ New York, NY, USA}
\email{sawin@math.columbia.edu}

\maketitle

\begin{abstract} Manin's conjecture predicts the number of rational points of bounded height on a Fano variety. To make this prediction precise, it is necessary to remove a thin subset of rational points. Peyre has tentatively proposed replacing this subset by the set of points where a certain freeness function he defined takes small values. We show that this proposal fails in the case of $\operatorname{Hilb}^2(\mathbb P^n)$, because the usual thin subset, consisting of rational points that lift to a certain double cover, contains many points with relatively large freeness. \end{abstract}

\section{Introduction}

Let $X$ be a geometrically integral smooth projective Fano variety over $\mathbb Q$ of dimension $n$ with Picard rank $r$. Let $\mathcal X$ be a proper integral model of $X$ over $\mathbb Z$. By fixing a volume form on $X(\mathbb R)$, we can define an anticanonical height function $H$ of rational points on $X$. The Manin-Peyre conjecture predicts the distribution of rational points of bounded height on $X$, both in terms of their number and their distribution among the adelic points of $X$.

To get a good statement, it is necessary to remove some thin sets of points on $X$. 

Formally, we say that a map $f: Y \to X$ of geometrically integral smooth projective varieties is a thin map if it is generically finite onto its image and its degree is not $1$.

The modern formulation of the conjecture combines work of Manin \cite{FMT}, Peyre \cite{peyre-duke}\cite{peyre-bordeaux}, Batyrev and Tschinkel \cite{BT}, and Salberger \cite{Salberger}.

\begin{conjecture}[Modern formulation of Manin's conjecture]\label{manin} There exists a finite set of thin maps $f_i: Y_i \to X_i$ such that
\[ \lim_{B \to \infty} \frac{  1 }{ B (\log B)^{r-1}} \sum_{ \substack{ x \in X(\mathbb Q) \\ H(x) < B \\ x \not\in f_i (Y_i(\mathbb Q)) \textrm{ for any }i}} \delta_x =  \alpha(X) \beta(X) \tau^{Br},  \]
where the (weak) limit is taken as measures on $X( \mathbb A_{\mathbb Q})$, $\delta_x$ is the measure of mass $1$ supported at $x$, $\tau^{Br}$ is the restriction to the subset of $X(\mathbb A_{\mathbb Q})$ where the Brauer-Manin obstruction vanishes of the Tamagawa measure \[\tau= \left( \lim_{s\to 1} (s-1)^r L (s, \operatorname{Pic} X_{\overline{\mathbb Q}})\right) \prod_v L_v ( s, \operatorname{Pic} X_{\overline{\mathbb Q}})^{-1} \omega_v \] with $\omega_v$ be the natural measure on $X(\mathbb Q_v)$ defined by the integral model $\mathcal X$ if $v$ is non-Archimedean or the volume form if $v=\infty$, \[\alpha(X) = r \operatorname{vol} \{  y \in  ((\operatorname{Pic}(X) \otimes \mathbb R)^{eff})^{\vee} \mid  K_X \cdot y\leq 1 \}  ,\] and \[\beta(X) = |H^1( \operatorname{Gal} (\overline{\mathbb Q}/\mathbb Q),\operatorname{Pic} X_{\overline{\mathbb Q}})|.\] \end{conjecture}

The minimal finite set of thin maps for which Conjecture \ref{manin} should be valid was given a purely geometric description in \cite[\S5]{LST}.

Peyre \cite{peyre-freedom}\cite{peyre-allheights} has proposed two notions, ``freeness" and the ``all the heights" approach, to replace the thin maps in Conjecture \ref{manin}. In this article, we will show that freeness cannot do the job alone. First we review the definition of freeness. 

Fix a rank $n$ vector bundle $\mathcal T$ on $\mathcal X$ that agrees with the tangent bundle away from finitely many primes, and fix a Riemannian metric on $T X_{\mathbb R}$. The determinant of $\mathcal T$ is an Arakelov line bundle structure on the anticanonical bundle of $X$, and therefore defines an anticanonical height function $H(X)$ on $K$. For convenience we will use this anticanonical height function. 

For $x$ a rational point of $X$, we can extend $x$ to a section of $\mathcal X$ defined over $\mathbb Z$.  This gives $ T_x { X}$ the structure of a rank $n$ vector bundle on $\mathbb Z$, with a metric. Such vector bundles are isomorphic to $\mathbb Z^n$, and so $T_x X$ may be viewed as a rank $n$ lattice.

Following Bost, Peyre has defined for a lattice $\Lambda$  of rank $n$ slopes $\mu_1(\Lambda),\dots, \mu_n(\Lambda)$ satisfying $\mu_1(\Lambda) \geq \dots \geq \mu_n(\Lambda)$,  $\sum_{i=1}^n \mu_i (\Lambda) = -\log \operatorname{vol} \mathbb R^n/\Lambda$, and $\mu_i (\Lambda)=- \log \gamma_i(\Lambda) + O_n(1)$ where $\gamma_i$ is the logarithm of the $i$th successive minimum of the lattice $\Lambda$ \cite[Definition 4.4]{peyre-freedom}. For us only the approximate value of freeness is relevant, but for clarity, to define freeness, first define \[ m_{\Lambda}' (k) = \sup \{ - \log \operatorname{vol} { \Lambda'} \mid \Lambda' \subseteq \Lambda, \operatorname{rank}(\Lambda') = k \] for $k \in \{0, \dots n\}$ and then define the convex hull as \[ m_\Lambda(i) = \sup \left\{   \frac{ (k_2- i) m_{\Lambda}' (k_1) + (i-k_1) m_{\Lambda}'(k_2) }{ k_2 - k_1} \mid k_1 \leq i \leq k_2 \right\}\] and \[ \mu_{i} (\Lambda)= m_{\Lambda}(i) - m_{\lambda} ( i-1 ) .\] Because $m_\Lambda$ is a piecewise linear function, linear on the interval $[i,i-1]$, $\mu_i$  is its slope on the interval, justifying the name.

We define the freeness of $x$ to be \cite[Definition 4.5]{peyre-freedom} \[ l(x) = \frac{ \max (   \mu_n( T_x { X}), 0 ) }{(  \log H(x))/n } .\]  Because $H(x) = - \log \operatorname{vol} ( T_x X)$, we have $\log H(x) = \sum_{i=1}^n \mu_i ( x^*  T_X) $ \cite[Remark 4.5(a)]{peyre-freedom}, and thus $ \mu_n ( x^*  T_{ X} ) \leq \log H(x)/n$ so $l(x) \leq 1$.

We fix a function $\epsilon(t) $ which goes to $0$ as $t$ goes to $\infty$, but does so slower than any power of $\log \log t$. 

Peyre suggested \cite[Empirical Formula 6.13 and Empirical Distribution 6.18]{peyre-freedom} that we may be able to replace the condition ``$x \not\in f_i (Y_i(\mathbb Q) )\textrm{ for any }i$" in Manin's conjecture with ``$l(x) > \epsilon(H(x))$". In this note, we show that this is not true.

Specifically, let $X = \Hilb^2 (\mathbb P^n)$ be the Hilbert scheme of pairs of points over projective space, defined over $\mathbb Q$, and let $\mathcal X$ be the corresponding scheme over $\mathbb Z$. The space $X$ is smooth, projective, geometrically integral, and, if $n>2$, Fano (Lemma \ref{fano}). All our other arguments, however, will require only that $n \geq 2$. 

 As long as $n\geq 2$, the Picard rank $r$ is $2$.
 
 We have a  double covering $f: Bl_{\Delta} (\mathbb P^n \times \mathbb P^n) \to \Hilb^2 (\mathbb P^n)$, where $Bl_{\Delta} (\mathbb P^n \times \mathbb P^n) $ is the blowup of the diagonal $\Delta$ of $\mathbb P^n \times \mathbb P^n$. To construct this covering, note that there is a map $(\mathbb P^n \times \mathbb P^n - \Delta) \to \Hilb^2 (\mathbb P^n)$ that sends a distinct pair of points to the ideal vanishing at those two points. To extend this map to the whole space, it is necessary to blow up the diagonal only once. (After blowing up once, the indeterminacy locus has codimension at least $2$ and so is a proper subset of the exceptional divisor, but the indeterminacy locus is invariant under $PGL_{n+1}$, and $PGL_{n+1}$ acts transitively on the exceptional divisor, so the indeterminacy locus is empty.)

 \begin{theorem}\label{main} For any $\epsilon< n/(n+1)$, there is a set $S_\epsilon$ of rational points on $\Hilb^2 (\mathbb P^n)$ such that 
 
 \begin{enumerate}
 
 \item All points in $S_\epsilon$ are the image under $f$ of points of $Bl_{\Delta} (\mathbb P^n \times \mathbb P^n) (\mathbb Q)$ .
 
\item All points in $x \in S$ have $l(x)> \epsilon$.
 
 \item The number of points in $S_\epsilon$ of height less than $B$ is at least a constant times $B \log B$.
 
 \end{enumerate}
 
 \end{theorem}
 
 It follows immediately that the freeness variant of Conjecture \ref{manin} is not satisfied, because Theorem \ref{main} implies that the density of the thin set $f(Bl_{\Delta} (\mathbb P^n \times \mathbb P^n) (\mathbb Q))$ is positive, which contradicts equidistribution by \cite[Theorem 1.2]{Browning-Loughran}.
   
 However, it is easy to check that Peyre's ``all the heights" proposed modification to Manin's conjecture does remove this bad set, and hence it is possible that the combination of these two modifications could replace the breaking thin maps. We summarize this idea briefly: On a variety of Picard rank $r$, we take $r$ line bundles $L_1,\dots, L_r$ which generate $\operatorname{Pic}(X)$, at least over $\mathbb Q$, and put an Arakelov structure on each, giving $r$ height functions $H_1,\dots, H_r$. Fix a compact subset $D$ of $(\mathbb R^{>0})^r$. Fix $u \in \mathbb R^r$, which when viewed using the basis $L_1,\dots, L_r$ as a linear form on $\operatorname{Pic}(X) \otimes \mathbb R$ lies in the interior of the dual of the effective cone. Rather than counting points of bounded height, we count points such that $(H_1(X)/ B^{u_1}, \dots, H_r(X)/ B^{u_r}) \in D$, in the limit as $B$ goes to $\infty$. Peyre asks \cite[Question 4.8 and 4.10]{peyre-allheights} whether the analogue of Conjecture \ref{manin} holds for this point count.
 
 For $X= \Hilb^2 (\mathbb P^2)$, we could take $L_1= K_X^{-1}$ and $L_2$ equal to $\mathcal O(E)$. Because $L_2$ is effective, we must have $u_2>0$. Thus as $B$ goes to $\infty$, we must have $H_2(X) \to \infty$. This is significant to us as our set $S_\epsilon$ will be defined such that a certain function $c$, which will be an Arakelov height function for $\mathcal O(E)$, is bounded. Because $H_2(X)$ is bounded on $S_\epsilon$, $S_\epsilon$ will not affect the equidistribution in the ``all the heights" model. 
 
It is possible that Peyre's notion of ``freeness" can replace the thin maps $f_i$ of degree $0$ in Conjecture \ref{manin}, while the ``all the heights" approach can replace the thin maps of degree $\geq 2$.

For $\Hilb^2 (\mathbb P^2)$, Conjecture \ref{manin} was proved (without the freeness modification) by Le Rudelier \cite{Rudelier}, and a function field analogue was proved by M\^{a}nz\u{a}\c{t}eanu \cite{Manzateanu}.

\begin{remark} It may be possible to prove Conjecture \ref{manin}, or its ``all the heights" modification, for $\Hilb^2(\mathbb P^n)$ for any $n$ by viewing it as a $\mathbb P^2$-bundle over the Grassmanian $G(2,n+1)$ parameterizing lines in $\mathbb P^n$, using known point-counting results on the Grassmanian, and using lattice-point counting results to count points on the fibers. \end{remark}
 
This research was conducted during the period the author served as a Clay Research Fellow. I would like to thank Emmanuel Peyre, Tim Browning, and Johan de Jong for helpful conversations.
 
 \section{Proofs}
 
 Consider the map $b: Bl_\Delta (\mathbb P^n \times \mathbb P^n) \to \mathbb P^n \times \mathbb P^n$. We give the spaces $Bl_\Delta (\mathbb P^n \times \mathbb P^n) , \mathbb P^n \times \mathbb P^n,$ and $\Hilb^2 ( \mathbb P^n)$, as well as the maps $b$ and $f$, their integral structures arising from the standard integral structure $\mathbb P^n_{\mathbb Z}$ on $\mathbb P^n$.

 Fix Riemannian metrics on the real points of $\mathbb P^n, \mathbb P^n \times \mathbb P^n$, $Bl_\Delta (\mathbb P^n \times \mathbb P^n)$, and $\Hilb^2 ( \mathbb P^n)$.
 
 Fix a constant $\delta$ with $0< \delta<1/2$ and a constant \[ C > \frac{1}{ \max _{ x_1, x_2 \in \mathbb P^n(\mathbb R) } d(x_1,x_2)} \] where the distance $d(x_1,x_2)$ is calculated using the fixed Riemannian metric. 

\begin{defi} For $x_1, x_2$ two distinct points in $\mathbb P^n(\mathbb Q)$, let \[ c(x_1,x_2) = \frac{ \max  \{ W \in \mathbb N \mid x_1 \equiv x_2 \mod W \}}{ d(x_1,x_2) } \] where we say $x_1 \equiv x_2 \mod W$ if $x_1$ and $x_2$ are equal when restricted to $\mathbb P^n(\mathbb Z/W)$. \end{defi}

\begin{defi} Let $S_{C,\delta}$ be the set of points in $\Hilb^2 (\mathbb P^n) (\mathbb Q) $ consisting of, for each $(x_1,x_2) \in \mathbb P^n(\mathbb Q)$ with $x_1\neq x_2$, $c(x_1,x_2) < C$, and $\log H(x_1), \log H(x_2) > \delta (\log H(x_1) + \log H(x_2)) $, the point $f ( b^{-1}(x_1,x_2)) $ (i.e. the ideal of functions vanishing at $x_1$ and $x_2$.) \end{defi}

\begin{lemma}\label{lattice-comparison}

 Let $Y$ and $X$ be schemes over $\mathbb Z$ that are smooth of dimension $n$, with proper generic fibers. Let $f: Y \to X$ be a map that is generically \'{e}tale. Fix Riemannian metrics on $Y$ and $X$. 

Let $s \in K_Y \otimes f^* K_X^{-1}$ be the section defined by the natural map $f^* K_X \to K_Y$, and fix an absolute value on $K_Y \otimes f^* K_X^{-1}$ over $\mathbb R$.

Let $y$ be a point in the \'{e}tale locus of $f$. Then

\begin{enumerate}

\item The natural map $df:   T_y Y \to  T_{(f(y)} X$ of integer lattices has cokernel of order $\prod_{p} |s(y)|_p^{-1}$.

\item For any element $u \in T_y Y$, we have \[ \lVert u\rVert |s(y)|_{\infty}\ll \lVert df(u) \rVert \ll  \lVert u \rVert  .\] 

\end{enumerate}

\end{lemma} 

Recall that the vanishing divisor of $s$ is called the ramification divisor of $f$.

\begin{proof} The determinant of $df$ is the map \[ \det(df)  y^*  K_Y^{-1} \to  (f(y))^* K_X^{-1}.\]  Thus $\det(df)$ may be viewed as a section of $y^* ( K_Y \otimes f^* K_X^{-1})$. By definition, this is $s(y)$. Hence $|s(y)|_p$ is the $p$-adic absolute value of the determinant of $df$. Because $df$ is an injective map of integral lattices, its determinant is the order of its cokernel.

For part (2), we first prove the inequality $ \lVert df(u) \rVert \ll  \lVert u\rVert $. This follows from the fact that $f$ is a differentiable morphism of compact manifolds, so the norm of its first derivative is continuous on a compact space, hence bounded. This implies that the singular values of $df$, viewed as a map of real vector spaces, are bounded. Because the product of the singular values is the absolute value of the determinant, this implies that the least singular value is bounded by a multiple of the determinant, which is $|s(y)|_{\infty}$. This gives the other inequality.

\end{proof}

Let us now apply this lemma to the maps $f$ and $b$. 

\begin{lemma}\label{ramification-divisor} The ramification divisor of $f$ is the exceptional divisor $E$ of $b$, and the ramification divisor of $b$ is $(n-1) E$. \end{lemma}

\begin{proof} The map $f$ is a double covering of smooth varieties, so its ramification divisor cannot have multiplicity. It can be viewed as the quotient of an involution swapping the two copies of $\mathbb P^n$, which fixes only points on $E$, so it is ramified only at $E$. It must ramify at $E$ because $\operatorname{Hilb}^2 (\mathbb P^n)$ is rationally connected. Thus the ramification divisor of $f$ is $E$. 

The map $b$ is a blowup of a smooth variety at a smooth subvariety of codimension $n$, and hence the ramification divisor is $(n-1)$ times the exceptional divisor.

\end{proof}

\begin{lemma}\label{slope-comparison} Let $y$ be a point of $Bl_\Delta (\mathbb P^n \times \mathbb P^n)( \mathbb Q)$ whose image under $b$ is a pair $(x_1,x_2)$ of distinct points of $\mathbb P^n(\mathbb Q)$. 

Then \[\left| \mu_i ( ( T_{b(y)} { \mathbb P^n \times \mathbb P^n} ) - \mu_i ( ( T_{(f(y)} { \Hilb^2 ( \mathbb P^n )} ) \right| \leq  n  \log ( c(x_1,x_2) ) + O(1)  .\]

\end{lemma}\begin{proof}   Let us apply Lemma \ref{lattice-comparison} to $f: Bl_{\Delta} ( \mathbb P^n \times \mathbb P^n) \to X$ .  We take $K_Y \otimes f^* K_X^{-1} = \mathcal O(E)$, $s$ the natural section, and the absolute value to be the standard absolute value on $\mathcal O$ times the pullback of the distance function $d(x_1,x_2)$ on $\mathbb P^n(\mathbb R) \times \mathbb P^n(\mathbb R)$. The distance function is an Arakelov metric on $\mathcal O(E)$ because it is nonvanishing away from $E$ and vanishes to order $1$ at $E$. 

Passing from a lattice to an index $N$ sublattice changes the slopes by at most $\log N$, and changing the metric on a lattice by a distortion factor of $\lambda$ changes the slopes by at most $\log \lambda$. By Lemma \ref{lattice-comparison}, the sum of these two contributions to \[\left| \mu_i ( (Y_y { Bl_{\Delta} ( \mathbb P^n \times \mathbb P^n) } ) - \mu_i ( T_{f(y)} { \Hilb^2 ( \mathbb P^n ) } ) \right| \] is at most $\log ( |s(y)|_{\infty}^{-1}  | \prod_p |s(y)|_p^{-1} ) + O(1)$. Since $s$ is a function that vanishes at the exceptional divisor, $|s(y)|_p$ is the $p$-adic distance from the exceptional divisor, which is exactly the $p$-adic distance between $x_1$ and $x_2$. Similarly, with our chosen norm, $s(y)_{\infty}$ is exactly $d(x_1,x_2)$, so \[ \log ( |s(y)|_{\infty}^{-1}  | \prod_p |s(y)|_p^{-1} ) + O(1) =  \log ( c(x_1,x_2) ) + O(1)  .\]

For  
\[\left| \mu_i ( (T_y { Bl_{\Delta} ( \mathbb P^n \times \mathbb P^n) } ) - \mu_i ( T_ {(b(y)} { \mathbb P^n \times \mathbb P^n} )  \right|\] the situation is identical except that the line bundle is raised to the $n-1$st power, which raises the norms to the same power and thus multiplies the logs by $n-1$.

Summing these terms, we get the stated bound. \end{proof} 

\begin{lemma}\label{height-comparison} For $y \in Bl_\Delta (\mathbb P^n \times \mathbb P^n)$ with $b(y) = (x_1,x_2)$ with $(x_1,x_2)$ distinct, we have \[\left|  \log H( f(y)) - \log H(x_1) + \log  H(x_2) \right |  = O (\log c (x_1,x_2) ) + O(1) .\] \end{lemma}

\begin{proof} This follows from Lemma \ref{slope-comparison} upon observing that the log of the height is the sum of the slopes and that the height of a point on $\mathbb P^n \times \mathbb P^n$ is the sum of the heights on $\mathbb P^n$.  \end{proof}

\begin{lemma}\label{freeness-lower-bound} All points in $S_{C, \delta} $ have freeness at least $\frac{2 \delta n}{n+1} - o_h(1) $. \end{lemma}

\begin{proof} By \cite[Corollary 7.4]{peyre-freedom}, the freeness of any point in projective spaces $\mathbb P^n$ is at least $n/(n+1)$. By \cite[Proposition 7.13]{peyre-freedom}, the freeness of $(x_1,x_2) \in \mathbb P^n \times \mathbb P^n$ is 

\[ 2n \frac{ \min ( l(x_1) \log H(x_1) /n , l(x_2) \log H(x_2)/n )}{\log H(x_1)+ \log H(x_2) } \geq 2  \frac{n}{n+1} \frac{\min( \log H(x_1), \log H(x_2)}{ \log H(x_1) + \log H(x_2) } \geq \frac{ 2 \delta n}{n+1}. \]

Then by Lemma \ref{slope-comparison}, when we take the inverse image along $b$ and the image along $f$, the slope and the height will both change by $O(1)$. Thus the freeness, which is the ratio of these two, will change by $o_{h}(1)$.

\end{proof}

\begin{lemma}\label{cardinality} For $B$ sufficiently large, the cardinality of the set of points of $S_{C,\delta}$ with height less than $B$ is at least a constant times $B \log  B$. \end{lemma}

\begin{proof} Because each pair of distinct points $(x_1,x_2) \in \mathbb P^n (\mathbb Q)$ has a single inverse image $y \in Bl_\Delta (\mathbb P^n \times \mathbb P^n)$, and the map $f$ to $\Hilb^2 (\mathbb P^n)$ is two-to-one, it suffices to prove this by counting pairs of points on $\mathbb P^n$. By Lemma \ref{height-comparison}, it suffices to replace the height condition by $\log H(x_1) +\log H(x_2) < \log H(x)$.

We can count points with heights in dyadic intervals. It suffices to show that the number of pairs of points $(x_1,x_2)$ with $B_1 < H(x_1) < 2B_1$, $B_2< H(x_2)< 2B_2$, $\log H_1, \log H_2> \delta (\log H_1 + \log H_2)$, $c(x_1,x_2)< C$ is at least a constant times $B_1B_2$. It suffices to restrict attention to pairs of points $x_1,x_2$ which are distinct mod $p$ for all $p$ and whose distance at $\infty$ is at most $C^{-1}$.

We can count using a standard sieve. The number of pairs of rational points on $\mathbb P^n$ that satisfy this condition is asymptotic to a constant times $H_1 H_2$. We must show a positive proportion satisfy the local conditions at each point. 

For each finite set of primes $S$ including $\infty$, the proportion of points that satisfy the conditions is equal to the proportion of adelic points that satisfy the condition, which is $\prod_{ p \in S} (1 - \frac{p-1} {p^{n+1}-1} )$ times the volume at $\infty$ of a nonempty open set, which is at least $\prod_p  (1 - \frac{p-1} {p^{n+1}-1} )$ times the volume. Because $n\geq 2$, this Euler product converges, giving an upper bound for the density. To show that the probability that a random point satisfies all these conditions is positive, it suffices to show that as $S$ goes to $\infty$ slowly with $H_1,H_2$, the density of $(x_1,x_2)$ which are congruent mod $p$ for some $p$ not in $S$ goes to zero. This is immediate from the Ekedahl sieve, because the set of pairs that are congruent mod $p$ has codimension $n \geq 2$.

\end{proof}

\begin{proof}[Proof of Theorem \ref{main}] Because $\epsilon < n/(n+1)$, we can find some $\delta$ with $(n+1) \epsilon/2n < \delta < 1/2$. We then fix any $C> \frac{1}{ \max _{ x_1, x_2 \in \mathbb P^n(\mathbb R) } d(x_1,x_2)}$ and define $S_{\epsilon}$ to be $S_{\delta, C}$ with all points of freeness $\leq \epsilon$ removed. By Lemma \ref{freeness-lower-bound} there are finitely many. Hence by Lemma \ref{cardinality}, condition (3) is satisfied. Conditions (1) follows from the definition of $S_{C,\delta}$ and condition (2) is automatic as we removed all points of lesser freeness. \end{proof}

\begin{lemma}\label{fano} If $n>2$, then $X$ is Fano. \end{lemma}

\begin{proof} It suffices to show that $f^* K_X^{-1}$ is ample, as $f$ is finite and surjective. By Lemma \ref{ramification-divisor}, \[ f^* K_X = K_{Bl_\Delta (\mathbb P^n \times \mathbb P^n) } \otimes \mathcal O(-E) = b^* K_{\mathbb P^n \times \mathbb P^n} \otimes \mathcal O ( (n-2) E ) = b^* \mathcal O( -n-1,-n-1) \otimes \mathcal O( n-2) E).\] Next note that $Bl_\Delta (\mathbb P^n \times \mathbb P^n) $ maps to the Grassmanian $G(2, n+1)$ of lines in $\mathbb P^n$, where we send two points to the line through them, which is well-defined on the exceptional divisor since we blowup. (In fact, this map factors through $X$). The pullback of the line bundle $\mathcal O(1)$ on the Grassmanian is $b^* \mathcal O(1,1) \otimes \mathcal O(-E)$, because we can represent sections (i.e. Pl\"{u}cker coordinates) as bilinear forms nonvanishing on the diagaonal.

Hence $f^* K_X$ is $b^* \mathcal O(3,3)$ times the pullback of $\mathcal O(n-2)$ from the Grassmanian. Because the map to $\mathbb P^n \times \mathbb P^n \times G(2,n+1)$ is injective, and this is the pullback of an ample line bundle along that map, it is ample.

\end{proof}

Thanks to Pieter Belmans for pointing out an error in an earlier version of the proof of Lemma \ref{fano}.

 \end{document}